\documentclass[12pt,reqno]{amsart}
\usepackage{amssymb,amsmath,amsthm} 
\usepackage[all,cmtip]{xy}

\usepackage{fullpage}

\numberwithin{equation}{section}

\theoremstyle{plain}
\newtheorem{thm}{Theorem}[section]
\newtheorem{lem}[thm]{Lemma}

\theoremstyle{definition}
\newtheorem{defn}[thm]{Definition}

\theoremstyle{remark}

\newcommand{\m}{\mathfrak{m}}
\renewcommand{\O}{\mathcal{O}}
\newcommand{\B}{\mathrm{B}}
\newcommand{\M}{\mathrm{M}}
\newcommand{\G}{\mathrm{G}}

\DeclareMathOperator{\im}{Im}

\DeclareMathOperator{\Spec}{Spec}

\DeclareMathOperator{\Gr}{Gr}
\DeclareMathOperator{\Fil}{Fil}

\DeclareMathOperator{\Frac}{Frac}

\DeclareFontFamily{U}{wncy}{}
\DeclareFontShape{U}{wncy}{m}{n}{<->wncyr10}{}
\DeclareSymbolFont{mcy}{U}{wncy}{m}{n}

\title{Semistable reduction for multi-filtered vector spaces}
\author{Shizhang Li}

\begin{document}

\maketitle

\section{Introduction}

In this paper, \(\O\) will be a valuation ring. We denote \(\Frac(\O)\) by \(K\)
and \(\O/\m\) by \(k\). For an \(\O\) module \(\M\), we denote \(\M_K=\M
\otimes_{\O} K\) and \(\overline{\M}=\M \otimes_{\O} k\). Also for an \(\O/\pi\)
module \(\widetilde{\M}\) we still denote \(\overline{\M}=\widetilde{\M}
\otimes_{\O/\pi} k\) for some \(\pi \in \m\). Unless stated otherwise, all the
modules considered in this paper will be finitely presented over corresponding
ring.

\begin{defn}
  A multi-filtered vector space \((V,\Fil^{\bullet}_i)\) over a field \(F\) is a
  finite dimensional \(F\)-vector space \(V\) together with finitely many
  filtrations \(\Fil^{\bullet}_i, 1 \leq i \leq n\), such that for all \(i\) the
  filtration \(\Fil^{\bullet}_i\) is
  \begin{enumerate}
  \item decreasing,
  \item indexed by natural numbers,
  \item exhaustive and
  \item separated:
  \end{enumerate}
  \[V=\Fil^0_i \supseteq \Fil^1_i \supseteq \cdots \supseteq \Fil^{l_i}_i=0.\]
  
  For any subspace (resp.\ quotient) \(W\) of \(V\), we give the obvious induced
  multi-filtration on them.
\end{defn}

\begin{defn}
  \begin{itemize}
  \item The weight of a multi-filtered vector space is
    \[w(V,\Fil^{\bullet}_i)=\sum_i \sum_j j \cdot \dim(\Fil_i^j/\Fil_i^{j+1}).\]
  \item The slope of a multi-filtered vector space is defined as
    \[\mu(V,\Fil^{\bullet}_i)=\frac{w(V,\Fil^{\bullet}_i)}{\dim(V)}.\]
  \item A multi-filtered vector space \((V,\Fil^{\bullet}_i)\) is called
    \emph{semistable} if for any nonzero subspace \(W \subseteq V\), we have
    \(\mu(W) \leq \mu(V)\).
  \end{itemize}
\end{defn}

If \(\M\) is an \(\O\)-lattice of a multi-filtered \(K\)-vector space \((\M
\otimes_{\O} K,\Fil^{\bullet}_i)\), we will still denote the induced filtration
\(\Fil^{\bullet}_i \cap \M\) on \(\M\) by \(\Fil^{\bullet}_i\). And we use the
notation \(\overline{\Fil^{\bullet}_i}\) (resp.\
\(\widetilde{\Fil^{\bullet}_i}\)) to denote the induced filtration on
\(\overline{\M}\) (resp.\ \(\widetilde{\M}=\M \otimes_{\O} \O/\pi\) for some
\(\pi \in \m\)). We call such a \(\M\) an \emph{integral model} of \((\M
\otimes_{\O} K,\Fil^{\bullet}_i)\).

The main theorem of this paper is the following analogue of Langton's theorem in
the setting of multi-filtered vector spaces.

\begin{thm}[Main Theorem]
\label{Main Theorem}
Let \(K\) be a valued field. For any semistable multi-filtered vector space
\((V,\mathrm{Fil}^{\bullet}_i)\), there exits an integral model
\((\M,\mathrm{Fil}^{\bullet}_i)\) such that its reduction \(\overline{\M}\)
with its induced filtration is again semistable.
\end{thm}

\section{Proof of the Main Theorem}
For technical reason, let us make the following definition.

\begin{defn}
  A submodule of \(\mathcal{O}^n\) (resp. \({({\mathcal{O}/\pi})}^n\)) is said
  to be saturated if the quotient is flat over \(\mathcal{O}\) (resp.
  \(\mathcal{O}/\pi\)).
\end{defn}

\begin{defn}
\label{setup}
  Let \(\M\) be an \(\O\)-lattice of a multi-filtered \(K\)-vector space \((\M
  \otimes_{\O} K,\Fil^{\bullet}_i)\) with its induced filtrations. Given a short
  exact sequence of \(k\)-vector spaces
  \[
    0 \to \overline{\B} \to \overline{\M} \to \overline{\G} \to 0
  \]
  we say it is \emph{liftable modulo \(\pi \in \m\)} if there exists a short
  exact sequence of \(\mathcal{O}/\pi\)-modules.
  \[
    0 \to \widetilde{\B} \to \widetilde{\M}=\M/(\pi \M) \to \widetilde{\G} \to 0
  \]
  such that
  \begin{enumerate}
  \item \(\widetilde{\G}\) is flat over \(\O/\pi\)
  \item this sequence reduces to the above sequence and
  \item \(\widetilde{\Fil^{j}_i} \cap \widetilde{\B}\) surjects onto
    \(\overline{\Fil^{j}_i} \cap \overline{\B}\).
  \end{enumerate}
\end{defn}  

\begin{lem}
\label{lift one filtration}
If there is only one filtration \(\mathrm{F}^{\bullet}\) on \(\mathrm{M}\), then
any short exact sequence \(0 \to \overline{\mathrm{B}} \to \overline{\mathrm{M}}
\to \overline{\mathrm{G}} \to 0\) is liftable, i.e., there exists short exact
sequence of \(\mathcal{O}\)-modules \(0 \to \mathrm{B} \to \mathrm{M} \to
\mathrm{G} \to 0\) such that
  \begin{enumerate}
  \item \(\mathrm{G}\) is flat over \(\mathcal{O}\)
  \item this sequence reduces to the above sequence
  \item \(\mathrm{F}^{\bullet} \cap \mathrm{B}\) surjects onto
    \(\overline{\mathrm{F}^{\bullet}} \cap \overline{\mathrm{B}}\).
  \end{enumerate}
  Moreover, fix one lifting and splitting \(\mathrm{M}=\mathrm{B} \oplus
  \mathrm{G}\), then all different liftings are graphs of morphisms in \(
  \mathfrak{m} \mathrm{Hom}_{\mathrm{F}^{\bullet}}(\mathrm{B},\mathrm{G})\).
\end{lem}

\begin{proof}
  Choose a basis of \(\overline{\mathrm{F}^{i}} \cap \overline{\mathrm{B}}\) and
  lift them inside \(\mathrm{F}^{i}\). Then we lift the rest of a basis of
  \(\overline{\mathrm{B}}\) arbitrarily. Let \(\mathrm{B}\) be the sum of
  respective liftings. By construction \(\mathrm{B}\) satisfies (2) and (3).
  Note that \(\mathrm{B}\) is a saturated \(\mathcal{O}\)-submodule in
  \(\mathrm{M}\), hence (1) is also satisfied.
  
  If one has another lifting \( 0 \to \mathrm{B'} \to \mathrm{M} \to \mathrm{G'}
  \to 0\), then the projection from \(\mathrm{B'}\) to \(\mathrm{B}\) must be an
  isomorphism as guaranteed by Nakayama's lemma. Then this lifting
  \(\mathrm{B'}\) is just a graph of some morphism \(f \in
  \mathrm{Hom}(\mathrm{B},\mathrm{G})\). (3) implies that this morphism must
  preserve induced filtration. (2) implies that this morphism must have target
  \(\mathfrak{m} \mathrm{G}\). So we see that \(f \in
  \mathrm{Hom}_{\mathrm{F}^{\bullet}}(\mathrm{B},\mathfrak{m}\mathrm{G}) =
  \mathfrak{m} \mathrm{Hom}_{\mathrm{F}^{\bullet}}(\mathrm{B}, \mathrm{G})\). By
  the same reasoning any such \(f\) would give rise to a lifting satisfying (1),
  (2) and (3).
\end{proof}

The second half of the lemma above can be generalized to multi-filtered
situation in the following way:

\begin{lem}
  In the situation of~\ref{setup}, suppose \( 0 \to \overline{\mathrm{B}} \to
  \overline{\mathrm{M}} \to \overline{\mathrm{G}} \to 0\) is liftable modulo
  \(\pi \in \mathcal{O}\), and fix a splitting
  \(\widetilde{\mathrm{M}}=\widetilde{\mathrm{B}} \oplus
  \widetilde{\mathrm{G}}\). Then all different liftings correspond to graphs of
  morphisms in
  \(\mathrm{Hom}_{\Fil^{\bullet}}(\widetilde{\mathrm{B}},\widetilde{\mathrm{G}})
  \cap \mathfrak{m}
  \mathrm{Hom}(\widetilde{\mathrm{B}},\widetilde{\mathrm{G}})\).
\end{lem}

It's nice to prove the following lemma directly:

\begin{lem}
\label{fg}
  In the situation of the lemma above,
  \(\mathrm{Hom}_{\mathrm{Fil^{\bullet}}}(\widetilde{\mathrm{B}},\widetilde{\mathrm{G}})\)
  is always a finitely generated \((\mathcal{O}/\pi)\)-module.
\end{lem}

\begin{proof}
  We will prove by induction on the number of filtrations. The homomorphisms
  preserve all the \(n\) filtrations are exactly those preserve the first
  \(n-1\) filtrations and preserve the last filtration. Hence as a submodule of
  \(\mathrm{Hom}(\widetilde{\mathrm{B}},\widetilde{\mathrm{G}})\) (which is
  abstractly isomorphic to \({(\mathcal{O}/\pi)}^n\)), we see that
  \[\mathrm{Hom}_{\mathrm{Fil^{\bullet}_1,\dots,Fil^{\bullet}_n}}(\widetilde{\mathrm{B}},\widetilde{\mathrm{G}})=\mathrm{Hom}_{\mathrm{Fil^{\bullet}_1,\dots,Fil^{\bullet}_{n-1}}}(\widetilde{\mathrm{B}},\widetilde{\mathrm{G}})
    \cap
    \mathrm{Hom}_{\mathrm{Fil^{\bullet}_n}}(\widetilde{\mathrm{B}},\widetilde{\mathrm{G}})\]
  By induction hypothesis we see that
  \(\mathrm{Hom}_{\mathrm{Fil^{\bullet}_1,\dots,Fil^{\bullet}_{n-1}}}(\widetilde{\mathrm{B}},\widetilde{\mathrm{G}})\)
  is a finitely generated submodule of
  \(\mathrm{Hom}(\widetilde{\mathrm{B}},\widetilde{\mathrm{G}})\) and by
  Lemma~\ref{lift one filtration}
  \(\mathrm{Hom}_{\mathrm{Fil^{\bullet}_n}}(\widetilde{\mathrm{B}},\widetilde{\mathrm{G}})\)
  is a saturated submodule of
  \(\mathrm{Hom}(\widetilde{\mathrm{B}},\widetilde{\mathrm{G}})\). Therefore it
  suffices to prove that the intersection of a finitely generated submodule
  \(H_1\) and a saturated submodule \(H_2\) of \(H={(\mathcal{O}/\pi)}^n\) is
  again finitely generated. Fix a splitting \(H=H_2 \oplus H_3\), we see that
  \(H_1 \cap H_2\) is the kernel of projection map \(H_1 \to H_3\). This is
  finitely generated because \(\mathcal{O}/\pi\) is a coherent ring.
\end{proof}

\begin{lem}
\label{zhihe}
  Let \(H={(\mathcal{O}/\pi)}^N\). Let \(H_1\) be a finitely generated submodule
  in \(H\) and let \(H_2\) be a saturated finitely generated submodule in \(H\).
  Then
  \[
    H/(H_1 \cap \mathfrak{m} H + \mathfrak{m} H_2)=\bigoplus \mathcal{O}/\pi_j
    \oplus \bigoplus \mathcal{O}/\pi_i'\mathfrak{m} \bigoplus
    {(\mathcal{O}/\mathfrak{m})}^{\oplus s},
  \]
  where the direct sum is finite and \(\pi_i', \pi_j \in \mathfrak{m}-\{0\}\).
\end{lem}

\begin{proof}
  Because \(H_2\) is saturated, we may assume that \(H=H_2 \oplus H_3\). Then we
  see \(H':= H/\mathfrak{m}H_2 = H_2/\mathfrak{m}H_2 \oplus H_3\), and let us
  fix a basis of \(H_3\) as \(\{f_i\}\). Notice that \(\overline{H_1}
  \cap \mathfrak{m}\overline{H}=\overline{H_1 \cap \mathfrak{m} H}\), where
  ``bar'' of a module denote its image in \(H/\mathfrak{m}H_2\). Let us consider
  the image of \(H_1\) (suppose there are \(L\) generators) inside \(H'\).
  Project further to \(H_2':=H_2/\mathfrak{m}H_2\), we may assume the image of
  the first \(l\) generators form a basis of the image in \(H_2'\). Without loss
  of generality we can now assume the image of the generators are of the form
  \(e_i+\sum a_{ij}f_j\) and \(\sum a_{ij}f_j\), where \(e_1,\ldots,e_l\) is the
  image of the first \(l\) generators.

  Now let us perform Gauss elimination carefully as following: choose \(a_{ij}\)
  with smallest valuation (if it is possible then we would prefer to a choice
  with \(i > l\)), then do Gauss elimination to cancel all the other \(f_j\)
  coordinates appearing in other generators. Keep doing this procedure. After
  rearranging basis of both \(H_2'\) and \(H_3\), we can assume the image of
  \(L\) generators of \(H_1\) in \(H'\) are of the form \(e_i + \pi_i'f_i\) and
  \(\pi_j f_j\). Now we can finally conclude that the quotient is of the form as
  stated in this lemma.
\end{proof}

\begin{lem}
\label{lift to the most}
In the situation of~\ref{setup}, there exists a \(\pi \in \mathfrak{m}\) with
biggest valuation such that the sequence can be lifted modulo \(\pi\).
\end{lem}

\begin{proof}
  Let us do induction on the number of filtrations. Suppose there exists a
  \(\pi'\) with biggest valuation such that the first \(n-1\) filtrations are
  lifted, say, \( 0 \to \widetilde{\mathrm{B}} \to
  \widetilde{\mathrm{M}}=\mathrm{M}/(\pi \mathrm{M}) \to \widetilde{\mathrm{G}}
  \to 0\). Let us fix a splitting
  \(\widetilde{\mathrm{M}}=\widetilde{\mathrm{B}} \oplus
  \widetilde{\mathrm{G}}\). By Lemma~\ref{lift one filtration}, the n-th
  filtration can be lifted modulo \(\pi\) also. By the same reasoning as in
  Lemma~\ref{lift one filtration}, these two liftings differ by an \(A \in
  \mathfrak{m}\mathrm{Hom}(\widetilde{\mathrm{B}}, \widetilde{\mathrm{G}}) :=
  \mathfrak{m}H\). Both of these two liftings are not unique. They differ by
  elements in \(H_1 \cap \mathfrak{m}H\) and \(H_2 \cap
  \mathfrak{m}H=\mathfrak{m}H_2\), where
  \(H_1=\mathrm{Hom}_{\mathrm{Fil_1,\ldots,Fil_{n-1}}}(\widetilde{\mathrm{B}},\widetilde{\mathrm{G}})\)
  which is finitely generated by Lemma~\ref{fg} and
  \(H_2=\mathrm{Hom}_{\mathrm{Fil_n}}(\widetilde{\mathrm{B}},\widetilde{\mathrm{G}})\)
  which is saturated. We see that by Lemma~\ref{zhihe}, \(H/(H_1 \cap
  \mathfrak{m} H + \mathfrak{m} H_2)=\oplus \mathcal{O}/\pi_k \bigoplus \oplus
  \mathcal{O}/\pi_k'\mathfrak{m}\) in which residue of \(A\) lives. The
  coordinates of residue of \(A\) generate a principal ideal \((\pi)\) in
  \((\mathcal{O}/\pi')\). Because \(A \in \mathfrak{m}H\) we see that \(\pi \in
  \mathfrak{m}\). It is easy to see that \(\pi\) satisfies the lemma.
\end{proof}

The following lemma will give another characterization of liftable modulo
\(\pi\).

\begin{lem}
\label{another characterization}
  Given \( 0 \to \overline{\mathrm{B}} \to \overline{\mathrm{M}} \to
  \overline{\mathrm{G}} \to 0\) in the situation of~\ref{setup}. Then the
  following are equivalent:
  \begin{enumerate}
  \item the sequence above is liftable modulo \(\pi\)
  \item there exists a splitting \(\mathrm{M}=\mathrm{B}\oplus\mathrm{G}\) that
    reduces to the above sequence and for all \(\overline{v} \in
    \overline{\mathrm{B}} \cap \overline{\mathrm{Fil_i}}\) there exists a
    lifting \(v \in \mathrm{Fil_i}\) such that \(p_2(v) \in \pi G \) where
    \(p_2\) is the obvious projection onto \(G\).
    \item the same as above except \( \overline{v} \) only have to run over a set of
      chosen basis of the induced filtration.
   \end{enumerate}
\end{lem}

\begin{proof}
  (2) and (3) are obviously equivalent by finiteness of dimension and linearity.
  
  (1) implies (2): liftable means one can find splitting
  \(\widetilde{\mathrm{M}}=\widetilde{\mathrm{B}}\oplus\widetilde{\mathrm{G}}\),
  lift this splitting over \(\mathcal{O}\) we see that (2) is fulfilled.

  (2) implies (1): one just modulo this lifting by \(\pi\).
\end{proof}

Now we are ready to prove the analogue of Langton's theorem for multi-filtered
vector spaces.

\begin{proof}[Proof of the Main Theorem]
  Choose an arbitrary integral model of \((V,\mathrm{Fil}^{\bullet}_i)\), let us
  denote it as \((\M,\mathrm{Fil}^{\bullet}_i)\). Then consider its reduction
  \(\overline{\M}\) with induced multi-filtrations. If it is semistable, then we
  are done. Otherwise, there is a maximal destabilizing subspace
  \(\overline{\mathrm{B}}\) inside \(\overline{\M}\) (cf.\cite[Lemma 2.5]{JHE}):
  \( 0 \to \overline{\mathrm{B}} \to \overline{\mathrm{M}} \to
  \overline{\mathrm{G}} \to 0\). By Lemma~\ref{lift to the most} there exists a
  \(\pi\) with biggest valuation, such that the sequence above is liftable
  modulo \(\pi\). Then \(\pi\) is not \(0\) because
  \((V,\mathrm{Fil}^{\bullet}_i)\) is semistable. Let us denote a lifting as \(
  0 \to \widetilde{\mathrm{B}} \to \widetilde{\mathrm{M}}=\mathrm{M}/(\pi
  \mathrm{M}) \to \widetilde{\mathrm{G}} \to 0\). This lifting gives us a
  splitting \(\widetilde{\mathrm{M}}=\widetilde{\mathrm{B}} \oplus
  \widetilde{\mathrm{G}}\), we can and do lift this splitting further to get
  \(\mathrm{M}=\mathrm{B} \oplus \mathrm{G}\). Let \(\M'=\ker(\M \to
  \widetilde{\mathrm{G}})\) be another integral model.

  Now we have the following sequence
  \[
    0 \to \frac{\pi \M}{\pi \M'}=\frac{\M}{\M'}=\widetilde{\mathrm{G}} \to \frac{\M'}{\pi
      \M'} \to \frac{\M'}{\pi \M}=\widetilde{\mathrm{B}} \to 0
  \]
  tensoring with \(\kappa\) gives
  \[
    0 \to \overline{\mathrm{G}} \to \overline{\mathrm{M'}} \to
    \overline{\mathrm{B}} \to 0
  \]
  After computing in the single filtered case we see that it is a sequence of
  multi-filtered vector spaces.

  We claim that there is no splitting \(\overline{\mathrm{B}} \to
  \overline{\mathrm{M'}}\) as multi-filtered vector spaces. Otherwise there
  exists such a splitting \(\overline{h}: \overline{\mathrm{B}} \to
  \overline{\mathrm{M'}} \). Then lift this splitting to a morphism \(h:
  \mathrm{B} \to \pi\mathrm{G}\). Now let us consider the splitting
  \(\mathrm{M}=\Gamma_{h} \oplus \mathrm{G}\). Let \(\overline{v} \in
  \overline{\mathrm{B}} \cap \overline{\mathrm{F_i}}\), then by assumption we
  can lift \(\overline{h}(\overline{v})\) to \(v \in \mathrm{F_i}\) and
  \(h(p_1(v))-p_2(v) \in \pi\mathfrak{m} G\). Hence under the new splitting
  \(\mathrm{M}=\Gamma_{h} \oplus \mathrm{G}\), \(p'_2(v) \in \pi'G\) for some
  \(\pi' \in \pi \mathfrak{m}\) where \(p'_2=h \circ p_1 - p_2\) is the new
  obvious projection onto \(G\). Now by Lemma~\ref{another characterization} we
  can lift the original sequence modulo \(\pi'\) for some \(\pi' \in \pi
  \mathfrak{m}\) contradicting our assumption that \(\pi\) has the biggest
  valuation modulo which one can lift the sequence.

  Our last claim is that the maximal destabilizing subspace \(\overline{B'}\) of
  \(\overline{\M'}\) has either slope or dimension less than \(\overline{B}\).
  Denote the image of \(\overline{B'}\) in \(\overline{B}\) by
  \(\im(\overline{B'})\). We have
  \[
    \mu (\overline{B'}) \leq \frac{w(\overline{B'} \cap
      \overline{G})+w(\im(\overline{B'}))}{\dim(\overline{B'})} \leq
    \frac{\dim(\overline{B'} \cap \overline{G}) \times \mu (\overline{B}) +
      \dim(\im(\overline{B'})) \times \mu (\overline{B})}{\dim(\overline{B'})} =
    \mu (\overline{B})
  \]
  where the first equality can be obtained only if \(\overline{B'} \cap
  \overline{G}=0\). In that case, we see immediately that \(w(\overline{B'})
  \leq w(\im(\overline{B'}))\) where equality is obtained only if
  \(\overline{B'}\) is a section of its image inside \(\overline{B}\) as a
  multi-filtered vector space. Therefore if
  \(\mu(\overline{B'})=\mu(\overline{B})\) then \(\overline{B'}\) is isomorphic
  to \(\im(\overline{B'})\) (as two multi-filtered vector spaces) which is a
  proper subspace of \(\overline{B}\) (by last paragraph). Hence in the
  situation of \(\mu(\overline{B'})=\mu(\overline{B})\) the dimension of
  \(\overline{B'}\) must be strictly less than that of \(\overline{B}\).

  Because the ranges of slope and dimension are finite, after doing the
  procedure above finitely many times, we see that the reduction would become
  semistable.
\end{proof}

\section{Stacks of semistable multi-filtered vector bundles}
\label{application}

In this section, let us make the following definition.

\begin{defn}
  Fix two natural numbers \(s\) and \(n\) and \(s\) non-decreasing functions
  \(l_i: \mathbb{N} \to \{0,\dots,n\}\) such that \(l_i(0)=0\) and
  \(l_i(\infty)=n\), we call such a datum \emph{type} and denote it as
  \((s,n;l_i)\).
\end{defn}

Given a type \((s,n;l_i)\), we can consider the stack
\(\mathcal{M}_{(s,n;l_i)}\) which associates any scheme \(X\) the groupoid of
multi-filtered rank \(n\) vector bundles of type \((s,n;l_i)\) on \(X\), i.e.,
there are \(s\) decreasing exhaustive separated filtrations \(\Fil_j^{\bullet}\)
satisfying the following two properties:
\begin{enumerate}
\item the graded pieces \(\Gr_i^j=\Fil_i^j/\Fil_i^{j+1}\) are flat over \(X\);
\item the ranks of \(\Fil_i^j\) are \(n-l_i(j)\).
\end{enumerate}

It is not hard to see that
\(\mathcal{M}_{(s,n;l_i)}=[\prod_i\mathrm{Fl}(n;l_i)/\mathrm{GL}_n]\) where
\(\mathrm{Fl}(n;l_i)\) is the flag variety with numerical conditions as above. Therefore the stack
\(\mathcal{M}_{(s,n;l_i)}\) is an algebraic stack and it is quasi-separated and
of finite type over \(\Spec(\mathbb{Z})\).

There is an open subset \({(\prod_i\mathrm{Fl}(n;l_i))}^{s.s.}\) in
\(\prod_i\mathrm{Fl}(n;l_i)\) whose points correspond to semistable
multi-filtered vector spaces. As being semistable is independent of choice of
basis we see that this open subset defines an open substack
\({\mathcal{M}_{(s,n;l_i)}}^{s.s.}=[{(\prod_i\mathrm{Fl}(n;l_i))}^{s.s.}/\mathrm{GL}_n]\)
in \(\mathcal{M}_{(s,n;l_i)}\) which is just the stack of rank \(n\)
multi-filtered vector bundles of type \((s,n;l_i)\) whose fibers are semistable.

\begin{thm}
  \({\mathcal{M}_{(s,n;l_i)}}^{s.s.}\) is a quasi-separated algebraic stack
  which satisfies the existence part of the valuative criterion (c.f.~\cite[Tag
  0CL9]{stacks-project}) over \(\Spec(\mathbb{Z})\).
\end{thm}

\begin{proof}
  From the discussion before theorem we know that
  \({\mathcal{M}_{(s,n;l_i)}}^{s.s.}\) is an algebraic stack quasi-separated and
  of finite type over \(\Spec(\mathbb{Z})\). By our Theorem~\ref{Main Theorem},
  we see the structure morphism \({\mathcal{M}_{(s,n;l_i)}}^{s.s.} \to
  \Spec(\mathbb{Z})\) satisfies the valuative criterion in~\cite[Tag
  0CL9]{stacks-project}.
\end{proof}

\section*{Acknowledgements}

The author wants to thank his advisor Johan de Jong for telling him this
elementary yet interesting problem.

\bibliographystyle{alpha}
\bibliography{Langton}

\end{document}